\documentclass[11pt,oneside,reqno]{amsart}
\usepackage{amssymb}
\usepackage{}
\usepackage{amsmath}
\usepackage{graphicx}
\usepackage{mathrsfs}
\usepackage{amsfonts}
\usepackage{enumerate,amsmath,amssymb,amsthm}
\usepackage{comment}

\usepackage{fancyhdr}
\usepackage{hyperref}
\hypersetup{backref,
colorlinks=true,
linkcolor=blue,
anchorcolor=blue,
citecolor=blue}

\numberwithin{equation}{section}

\usepackage{bm}
\usepackage{mathrsfs}
\usepackage{amsfonts}
\usepackage{color}
\usepackage{mathrsfs}
\usepackage{mathtools}
\usepackage{stmaryrd}
\usepackage{amsmath}
\usepackage{amssymb}
\usepackage{bbm}
\usepackage{mathrsfs}
\usepackage{graphicx}
\usepackage{enumerate}
\usepackage{amsthm}
\newcommand{\ba}{\begin{eqnarray}}
\newcommand{\ea}{\end{eqnarray}}

\usepackage{amssymb}
\usepackage{mathtools}
\usepackage{stmaryrd}
\usepackage{amsmath}
\usepackage{amssymb}
\usepackage{bbm}
\usepackage{mathrsfs}
\usepackage{graphicx}
\usepackage{enumerate}

\newtheorem{theorem}{Theorem}[section]

\newtheorem{assumption}[theorem]{Assumption}
\newtheorem{definition}[theorem]{Definition}
\newtheorem{lemma}[theorem]{Lemma}
\newtheorem{remark}{Remark}[section]

\def\EE{\mathbb{E}}

\def\d{{\mathrm{d}}}
\def\RR{\mathbb{R}}
\def\bmA{{\mathbf{A}}}

\allowdisplaybreaks

\title[ The stochastic MHD equations driven by pure jump noise in $L^p$ spaces]{{ The stochastic MHD equations driven by pure jump noise in $L^p$ spaces$^\dagger$}
}
  \thanks{$\dagger$  This work is partially supported by National Natural Science Foundation of China (No. 12071433).}
   \thanks{$*$ Corresponding author}
\begin{document}

\maketitle
\centerline{\scshape Kaicheng Ni$^{a}$,  Heling Su$^{a}$, Jiahui Zhu$^{a,*}$ }
\medskip
 {\footnotesize
\centerline{ $a.$   School of Mathematical Sciences, Zhejiang University of Technology, Hangzhou 310019, China }

\begin{abstract}
We consider the stochastic incompressible magnetohydrodynamic equations driven by additive jump noises  on either the whole space  $\mathbb{R}^d$, $d=2,3$ or  a smooth  bounded domain $D$ in $\mathbb{R}^d$. We establish the local existence and uniqueness of a mild solution in the space $L^q(0,T;\mathbb{L}^{p\otimes}_{\sigma}(D))$ allowing for initial data with less regularity, including the marginal case $u_0\in \mathbb{L}^{d\otimes}_{\sigma}(D)$. In the two-dimensional case, we also prove the global existence of mild solutions.

\hspace{1cm}

\textit{Keywords}:  Stochastic MHD system; mild solution;  compensated Poisson random measure; well-posedness

\textit{Mathematics Subject Classification}: 60H15, 35R60,   60G51
\end{abstract}

\section{Introduction}
  In this paper, we consider the following stochastic motion for incompressible Magneto-hydrodynamics (MHD) flows in $D\subset \RR^d$, $d=2,3$:
	\begin{align}
	\begin{split}\label{MHD-eq} 
		&\frac{\partial v}{\partial t}+(v\cdot\nabla)v-\Delta v - (H\cdot\nabla)H +\nabla\Bigl(p+\frac{|H|^2}{2}\Bigr)
		=\int_Z \xi_1(t,z)\tilde{N}_1(\mathrm{d}t,\mathrm{d}z) ,\\
		&\frac{\partial H}{\partial t}+(v\cdot\nabla)H+\text{curl}(\text{curl} H) - (H\cdot\nabla)v
		=\int_Z \xi_2(t,z)\tilde{N}_2(\mathrm{d}t,\mathrm{d}z),\\
		&\operatorname{div} v=0, \quad \operatorname{div}H=0,\\
		&v(x,0)=v_0(x),\quad H(x,0)=H_0(x),
		\end{split}
	\end{align}
where $D$ is either all of $\RR^d$ or  a  bounded domain in $\RR^d$ with compact and sufficiently smooth boundary. In this problem, $H, v$ and $p$ denote the magnetic field, velocity and pressure of the fluid respectively,  functions  $\xi_i:[0,T]\times \Omega\times Z_i\rightarrow \mathbb{L}_{\sigma}^{r}(D)$, $r\geq d$ are two measurable maps such that $\EE \int_0^T\int_{Z_i}\|\xi_i(s,z)\|_{ \mathbb{L}_{\sigma}^{r}(D)}^2\nu_i(\d z)\d s<\infty$, and $\tilde{N}_i(d t, d z)=N_i(\mathrm{d} t, \mathrm{d} z)-\nu_i(\mathrm{d} z) \mathrm{d} t$, $i=1,2,$ are two  compensated Poisson random measures with non-negative $\sigma$-finite intensity measures $\nu_i$ on measurable spaces $(Z_i,\mathcal{Z}_i)$, $i=1,2$. 
When $D$ is a bounded domain, equation \eqref{MHD-eq} is equipped with the boundary conditions $v=0 $, $H \cdot n=0$ and $\operatorname{curl} H \times \bm{n}=0$ on $\partial D$, where $\bm{n}$ is the unit outer normal on the boundary.

In this paper, we will prove the local existence and uniqueness of an $L^q([0, T) ; \mathbb{L}^{p\otimes}_{\sigma}(D))$-valued solution to Equation \eqref{MHD-eq}  for less regular initial data $L^{r\otimes}_{\sigma}(D)$, $r\geq d$, where $p$ and $q$ are chosen such that the $L^q([0, T) ; \mathbb{L}^{p\otimes}_{\sigma}(D))$-norm of the solution is scaling dimensionless, 
  i.e. $\frac{1}{q}=(\frac{1}{r}-\frac{1}{p})\frac{d}{2}$, $r<p<\infty $, see Theorem \ref{main}. 
 In the case of 2D domains, we also establish the global existence and uniqueness of an $L^q([0, T) ; \mathbb{L}^{p\otimes}_{\sigma}(D))$-valued mild solution, $p>2$ for  $\bm{u}_0 \in \mathbb{L}^{2\otimes}_{\sigma}(D)$, see Theorem \ref{Th-2D}.  Our results generalize those found in \cite{Sun-10, Zeng} for stochastic MHD equations,  and \cite{Fang, Zhu+Brz+Liu-19-SPL} for stochastic Navier-Stokes equaitons.  In \cite{Zeng}, Zeng prove that there is a local solution $\bm{u}$ in $C([0, T) ; \mathbb{L}^{p\otimes}_{\sigma}(D)$ for $\bm{u}_0 \in \mathbb{L}^{p\otimes}_{\sigma}(D):=\mathbb{L}^{p}_{\sigma}(D)\times \mathbb{L}^{p}_{\sigma}(D))$ if $p>d$.  
However, the proofs of \cite{Zeng}  can not directly apply to the marginal case $\bm{u}_0 \in \mathbb{L}^{d\otimes}_{\sigma}(D)$ and specialized techniques are required in $L^q([0, T) ; \mathbb{L}^{p\otimes}_{\sigma}(D))$.   To illustrate our results, let us recall the scaling property of MHD system: if $\bm{u}(x, t)=(v(x,t),H(x,t))$ solves \eqref{MHD-eq}, then for each $\lambda>0$,
$\bm{u}_\lambda(x, t)=\lambda \bm{u}(\lambda x, \lambda^2 t) $ 
also solves  \eqref{MHD-eq}.  
Observe that the $\mathbb{L}^{d\otimes}_{\sigma}(D)$-norm of the initial data has zero scaling dimension and the $L^q([0, T) ; \mathbb{L}^{p\otimes}_{\sigma}(D))$-norm of the solution also has zero scaling dimension.  In the marginal case, this scaling dimensionless property is critical as 
it guarantees the uniqueness of the solutions within the specified $L^q([0, T) ; \mathbb{L}^{p\otimes}_{\sigma}(D))$ function space. 
It is worth mentioning that the class $C([0, T;\mathbb{L}^{d\otimes}_{\sigma}(D))$ is usually insufficient to ensure the uniqueness. One may conduct \cite{Ni+Sacks} for an example of non-uniqueness in initial boundary value problems of nonlinear heat equations.

Compared to \cite{Fang, Sun-10, Zeng}, which considered fractional Brownian motion, our work deals with the jump noise, and this requires specialized techniques and tools. 
To handle the stochastic term of jump type, the ideas of our proofs are inspired by a recent work  \cite{BLZ} by Brze\'zniak, Liu and the third named author. 
Since the semigroup $\bm{S}(t)$ of our equation is only a bounded analytic $C_0$-semigroup in $\mathbb{L}^{p\otimes}_{\sigma}(D)$, we cannot apply the maximal inequality for $C_0$-contraction semigroups in \cite{Zhu+Brz+Liu-2019}. Instead we employ techniques from \cite{BLZ} and \cite{Zhu+Brz+Liu-19-SPL}.  The main feature of our approach here is that we can prove the trajectories of the corresponding Ornstein-Uhlenbeck process belong to the space $\mathbb{L}^q(0, T ; \mathbb{L}^{p\otimes}_{\sigma}(D))$ using the Burkholder inequality in Banach spaces.  Another novelty of the present paper is that our results are applicable to general, possibly unbounded domains. 
  This is important when examining the asymptotic behavior of solution.

The structure of the paper is as follows. In the next section, we introduce the functional setting of the problem and present the abstract formulation of the stochastic model. In section \ref{sec-lcoal-ex}, we prove the local existence and uniqueness of the mild solution in $L^q(0,T;\mathbb{L}_{\sigma}^{p\otimes}(D))$ and also establish global existence of mild solutions in two dimension.

\section{Preliminaries}\label{section-oper-non}

 For function spaces of $d$-dimensional vector valued functions, we use the blackboard bold letters, for instance,
$
\mathbb{L}^p(D)=\{u=\left(u_1,\ldots, u_d\right) \mid u_j \in L^p(D), j=1, \ldots, d\} .
$
Likewise for $\mathbb{C}_c^{\infty}(D)$ and $\mathbb{W}^{m, p}(D)$. Let 
$\mathbb{L}_{\sigma}^p(D)$ be the completion in $\mathbb{L}^p(D)$  of the set $\{u \in \mathbb{C}_c^{\infty}(D) \mid \nabla \cdot u=0 \text { in } D\}$.

We define the Stoke operator $A=-P \Delta$ with domain $\mathscr{D}(A)=\mathbb{L}_\sigma^p(D) \cap \mathbb{W}^{2, p}(D) \cap \mathbb{W}_0^{1, p}(D)$, where $P$ is the  Helmholtz projection  from  $\mathbb{L}^p(D)$ onto $ \mathbb{L}_{\sigma}^p(D)$. 
According to \cite{Giga-83},  $A$ is a closed densely defined linear operator and has a bounded inverse in $\mathbb{L}_{\sigma}^p(D)$.  The operator $-A$ generates a bounded analytic $C_0$-semigroup $(e^{-tA})_{t\geq 0}$ in $\mathbb{L}^{p }_{\sigma}(D)$.   
When $D=\mathbb{R}^d$, $A=-P \Delta=-\Delta P$, then $A$ is essentially equal to $-\Delta$ and $(e^{-tA})_{t\geq 0}$ is essentially the heat operator.

Define $ \mathcal{M}\, H =\mathrm{curl} \text { curl } H$ with domain $ \mathscr{D}(\mathcal{M}) =\mathbb{L}_\sigma^p(D) \cap\{H \in \mathbb{W}^{2, p}(D) \mid \bm{n}\times \operatorname{curl} H=0 \text { on } \partial D\}$. 
It is also shown in \cite{Miya-80, Miya-82} that $\mathcal{M}\,H \in \mathbb{L}_\sigma^p(D)$ for $H \in \mathscr{D}(\mathcal{M})$, the domain of $\mathscr{D}(\mathcal{M})$ is dense in $\mathbb{L}_\sigma^p(D)$, and $\mathcal{M}$ is a closed operator for $1<p<\infty$.  
Moreover, $\mathcal{M}$ is a sectorial operator.

We denote $\mathbb{L}^{p \otimes}(D)=\mathbb{L}^p (D)\times \mathbb{L}^p(D)$ and $\mathbb{L}_\sigma^{p \otimes}(D)=\mathbb{L}_\sigma^p(D) \times \mathbb{L}_\sigma^p(D)$. Define  the operator $\bmA: \mathscr{D}\left(\bmA\right) \subset \mathbb{L}_{\sigma}^{p\otimes}(D) \rightarrow \mathbb{L}^{p\otimes}_{\sigma}(D)$ by the formula $\bmA \bm{u}:
=(Av,\mathcal{M}\,H),$ $\bm{u}=(v,H)\in\mathscr{D}\left(\bmA\right)=\mathscr{D}(A)\times \mathscr{D}(\mathcal{M})$. 
Since both $A$ and $\mathcal{M}$ are sectorial operators, according to 
 \cite[Theorem 2.16]{Yagi-09}, the diagonal matrix operator $\mathbf{A}$, considered with the domain $\mathscr{D}(A) \times \mathscr{D}(\mathcal{M})$ is still a sectorial positive operator on the space $\mathbb{L}^{p\otimes}_{\sigma}(D)$. Let $(\bm{S}(t))_{t\geq 0}$  be the bounded analytic semigroup generated by $\bmA$  on $\mathbb{L}^{p\otimes}(D)$. Then for $\alpha>0$,  the fractional operator $\bmA^\alpha$ commutes with $\bm{S}(t)$ on $\mathscr{D}(\bmA^\alpha)$. By the properties of analytic semigroup, see \cite{Lunardi} (see also \cite[Lemma 2.2]{Giga-83}),  for any $\bm{u}=(v,H)\in \mathbb{L}_{\sigma}^{p\otimes}(D)$,
\begin{align}\label{ineq-semi-pro-1-MHD}
	\left\|\bmA^\alpha \bm{S}(t) \bm{u}\right\|_{\mathbb{L}_{\sigma}^{p\otimes}(D)} \leq C_\alpha\, t^{-\alpha}\|\bm{u}\|_{\mathbb{L}_{\sigma}^{p\otimes}(D)}.
\end{align}
Note that $\mathscr{D}(A^{\gamma})$ is continuously embedded in $\mathbb{H}^{2\gamma,r}(D)$ (c.f. \cite[Proposition 1.4]{Giga+Miya} ) and $\mathscr{D}(\mathcal{M}^{\gamma})$ is continuously embedded in $\mathbb{H}^{2\gamma,r}(D)$ (c.f. \cite[Theorem 16.2]{Yagi-09}).  By using the Sobolev embedding theorem 
and the regularity \eqref{ineq-semi-pro-1-MHD}, we deduce that for any $t>0$, $\bm{S}(t):\mathbb{L}^{r\otimes}_{\sigma}(D)\rightarrow \mathbb{L}^{p\otimes}_{\sigma}(D)$ is bounded and for any $\bm{u}\in \mathbb{L}^{r\otimes}_{\sigma}(D)$, 
	\begin{align}\label{MHD-semi-Lp-est}
	\left\| \bm{S}(t) \bm{u} \right\|_{\mathbb{L}_{\sigma}^{p\otimes}(D)} \leq C_{p,r}\, t^{-\frac{d}{2}\,\big(\frac1r-\frac1p\big)}\|\bm{u}\|_{\mathbb{L}_{\sigma}^{r\otimes}(D)} , \quad 1<r\leq p<\infty.
	\end{align}
	Throughout the paper, the symbol $C$ will denote a positive generic constant whose value may change from place to place. If a constant depends on some variable parameters, we will put them in subscripts.

Let us recall the Marcinkiewicz spaces (or weak Lebesgue spaces) $L^{(q)}(0,T)$ for $1\leq q\leq \infty$. 
The Marcinkiewicz space  $L^{(q)}(0, T)$, $1\leq q<\infty$, consists of all measurable functions  $f:(0,T)\rightarrow \mathbb{R}$ such that
$$
[f]_{L^{(q)}(0, T)}=\sup _{\lambda >0} \lambda \mu_f(\lambda)^{1 / q}<\infty.
$$
with $
\mu_f(\lambda)=\mu(\{t \in (0,T):|f(t)|>\lambda\}) 
$ being its distribution function  on $[0, \infty)$, where $\mu$ is the Lebesgue measure on $\mathbb{R}$. 
In the limiting case $q=\infty$,   $L^{(\infty)}(0, T)$, consists of all measurable function $f$ such that
$$
[f]_{L^{(\infty)}(0, T)}=\inf\{\lambda: \mu_f(\lambda)=0\}<\infty.
$$
Note that Marcinkiewicz spaces are  special cases of the more general  Lorentz spaces.  By using the  Marcinkiewicz spaces and generalized Marcinkiewicz interpolation theorem, we have the following Lemma whose proof being in the spirit of \cite{Giga-86}.  
This lemma allows us to consider the marginal case $r=d$.
\begin{lemma}\label{lemma-est-u_0} For $\bm{u}\in \mathbb{L}_{\mathrm{\sigma}}^{r\otimes}(D)$,   we have
\begin{align}
		\int_0^t\left\|\bm{S}(s) \bm{u}\right\|_{\mathbb{L}_{\sigma}^{p\otimes}(D)}^q \mathrm{d}  s \leq C\|\bm{u}\|_{\mathbb{L}_{\sigma}^{r\otimes}(D)}^q, \quad 0 \leqslant t \leqslant T ,
		\end{align}
with $\frac{1}{q}=(\frac{1}{r}-\frac{1}{p}) \frac{d}{2}, \;q> r>1$ and $C=C(p, q, r)$.
\end{lemma}

\begin{proof}

Consider the map $\Gamma$ defined by $\Gamma v=\|\bm{S}(t)  v\|_{\mathbb{L}_{\sigma}^{p\otimes}(D)}$ from $\mathbb{L}^{r\otimes}_{\sigma}(D)$ to functions on $[0, T)$. In view of \eqref{MHD-semi-Lp-est}, we have
\begin{align*}
\Big[\Gamma v\Big]_{L^{(q)}(0,T)} = \Big[\left\|\bm{S}(t)  v\right\|_{\mathbb{L}_{\sigma}^{p\otimes}(D) }\Big]_{L^{(q)}(0,T)} \leq  \Big[  C_{p,r}\, t^{-\frac{d}{2}\left(\frac1r-\frac1p\right)}\|v\|_{\mathbb{L}_{\sigma}^{r\otimes}(D)}   \Big]_{L^{(q)}(0,T)} =C_{p,r}\|v\|_{\mathbb{L}_{\sigma}^{r\otimes}(D)},
\end{align*}
 since $\frac{1}{q}=\left(\frac{1}{r}-\frac{1}{p}\right) \frac{d}{2}$ and $[t^{-\theta}]_{L^{(p)}}=1$ if $\theta p=1$. Then $\Gamma$ is continuous from $\mathbb{L}_{\sigma}^{r\otimes}(D)$ into $L^{(q)}(0,T)$. 
Since $|\Gamma v|= \left\|\bm{S}(t) v\right\|_{\mathbb{L}_{\sigma}^{p\otimes}(D)} \leq C_p\|v\|_{\mathbb{L}_{\sigma}^{p\otimes}(D)}$,  we have
$
     \big[\Gamma v\big]_{L^{(\infty)}(0,T)}\leq   C_{p}\|v\|_{\mathbb{L}_{\sigma}^{p\otimes}(D)}.
$
Then $\Gamma$ is continuous from $\mathbb{L}^{p\otimes}_{\sigma}(D)$ into $L^{(\infty)}(0,T)$. Now applying the Marcinkiewicz interpolation theorem, see e.g. \cite[Theorem 1.3.1]{BL-Interpolaiton}, we infer that $\Gamma$ is continuous from $\mathbb{L}_{\sigma}^{r_1\otimes}(D)$ into $L^{q_1}(0,T)$, for $r_1< q_1$ with $\frac1{q_1}=\big(\frac{1}{r_1}-\frac{1}{p}\big) \frac{d}{2}$.

\end{proof}

We define the following bilinear operator, for $\bm{u}_i=(v_i,H_i)^{T}\in \mathbb{C}_c^{\infty}(D)\times \mathbb{C}_c^{\infty}(D)$, $i=1,2$,
 $$\bm{B}(\bm{u}_1, \bm{u}_2):=\big( B_1(\bm{u}_1,\bm{u}_2) ,  B_2(\bm{u}_1,\bm{u}_2)\big)^T ,$$
with $B_1(\bm{u}_1, \bm{u}_2)= P(H_1 \cdot \nabla H_2)-P(v_1 \cdot \nabla v_2)$ and $B_2(\bm{u}_1,\bm{u}_2)=H_1\cdot \nabla v_2- v_1 \cdot \nabla H_2$.

 Consider the  trilinear form   
$
b(u, w, v)=\int_{D}(u \cdot \nabla w) v\, \d x,
$
for $ u \in \mathbb{C}_c^{\infty}(D), w, v \in \mathbb{C}_c^{\infty}(D)$. Note that if $\operatorname{div} u=0$, then 
$\sum_{i=1}^d \frac{\partial}{\partial x_i}(u_i w)=u \cdot \nabla w.$ We have $b(u, w, v)= -b(u, v, w) $ and for  $b(u, v, v)=0$, $u,w,v \in  \mathbb{C}_c^{\infty}(D)$ with $\operatorname{div} u=0$.  
Thus if $u \in \mathbb{C}_c^{\infty}(D)$, $\operatorname{div} u=0$, and $w, v \in \mathbb{C}_c^{\infty}(D)$, then by using the Hölder's inequality there exists a constant $c>0$ such that
\begin{align}\label{est-tri-b-1}
|b(u, w, v)| & =|b(u, v, w)|
  \leq\|u\|_{\mathbb{L}^4(D)}\|w\|_{\mathbb{L}^4(D)}\|\nabla v\|_{\mathbb{L}^2(D)} \leq\|u\|_{\mathbb{H}^1(D)}\|w\|_{\mathbb{H}^1(D)}\| v\|_{\mathbb{H}^1(D)}.
\end{align}
where we used the Sobolev embedding $\mathbb{H}^1(D) \subset \mathbb{L}^4(D)$, for $d=2,3$. 
Thus $b$ can be extended to a trilinear continuous form on $\mathbb{H}^1_{\sigma}(D)\times \mathbb{H}^1(D)\times \mathbb{H}^1(D)$.

Let
$(\Omega,\mathcal{F},\mathbb{F},\mathbb{P})$, where $\mathbb{F}=(\mathcal{F}_t)_{t\geq0}$,
be a filtered probability space satisfying the usual hypothesis. Let $\mathcal{P}$ be a predictable $\sigma$-field on $[0,T]\times\Omega$. 
Let $\tilde{N}_i(d t, d z)=N_i(\mathrm{d} t, \mathrm{d} z)-\nu_i(\mathrm{d} z) \mathrm{d} t$, $i=1,2,$ be two  compensated Poisson random measures with non-negative $\sigma$-finite intensity measures $\nu_i$ on measurable spaces $(Z_i,\mathcal{Z}_i)$, $i=1,2$.
\begin{assumption}\label{assu-xi} Let $r\geq d$. 
Assume that $\xi_i:[0,T]\times \Omega\times Z_i\rightarrow \mathbb{L}_{\sigma}^{r}(D)$ are two $\mathcal{P} \otimes \mathcal{Z}$-measurable maps such that $$\EE \int_0^T\int_{Z_i}\|\xi(s,z)\|_{ \mathbb{L}_{\sigma}^{r}(D)}^2\nu_i(\d z)\d s<\infty.$$
\end{assumption}

Let $Z:=Z_1\times Z_2$, $\mathcal{Z}:=\mathcal{Z}_1\otimes \mathcal{Z}_2$, $\nu:=\nu_1\otimes \nu_2$ and $N:=(N_1,N_2)$.   Let us define the map $\xi(t,\omega,z):=\big(\xi_1(t,\omega,z_1), \xi_2(t,\omega,z_2)\big)$, where $t\in[0,T]$, $z=(z_1,z_2)\in Z$, $\omega\in\Omega$. Then $\xi:[0,T]\times\Omega\times Z \rightarrow  \mathbb{L}_{\sigma}^{r\otimes}(D)$ is a $\mathcal{P} \otimes \mathcal{Z}$-measurable map and $\mathbb{E} \int_0^T\int_Z\|\xi(t, z)\|_{\mathbb{L}^{r\otimes}_{\sigma}(D)}^2 \nu(\mathrm{d} z) \mathrm{d} t<\infty .$

Working with the sectorial positive operator $\bmA$ we rewrite equation \eqref{MHD-eq}  by the following Cauchy problem of abstract evolution equations in the Banach space $\mathbb{L}_\sigma^{r\otimes}(D)$:
\begin{align}\label{MHD-eq-main}
\mathrm{d} \bm{u}(t)=[-\bmA\bm{u}(t)+\bm{B}(\bm{u}(t),\bm{u}(t))] \mathrm{d} st+\int_Z \xi(t, z) \tilde{N}(\d t, \d z),\; t\in(0,T],\quad \bm{u}(0)=\bm{u}_0.
\end{align}

\begin{remark}
According to \cite[Appendix B]{Brz95},  the spaces $\mathbb{L}_\sigma^{r\otimes}(D)$ with $2\leq r<\infty$ are martingale type 2 Banach spaces.  Under Assumption \ref{assu-xi}, the stochastic integral process $\int_0^t \int_Z \xi(s, z) \tilde{N}(\d s, \d z) $, $t\in[0,T]$ is well defined  and it is an $\mathbb{L}_\sigma^{r\otimes}(D)$-valued c\`adl\`ag $\mathbb{F}$-martingale, see \cite{BE09, Zhu+Brz+Liu-2019} for more details. 
\end{remark}

Note that in the above formulation $\bm{u}$ denotes the transpose of $(v, H)$ and the initial condition $\bm{u}_0=(v_0,H_0)$.

\begin{definition}
An $\mathbb{L}^{r\otimes}_{\sigma}(D)$-valued adapted process $\bm{u}(t), t \in[0, T]$, is a mild solution to \eqref{MHD-eq-main}  if 
$\bm{u} \in L^q(0, T ; \mathbb{L}_{\sigma}^{p\otimes}(D))$, $\mathbb{P}$-a.s. 
and for all $t \in[0, T]$, the following equality holds $\mathbb{P}$-a.s.
\begin{align}\label{local solution}
\bm{u}(t)=\bm{u}_0+\int_0^t \bm{S}(t-s) \bm{B}(\bm{u}(s), \bm{u}(s))\, \mathrm{d} s+\int_0^t \bm{S}(t-s) \xi(s, z)\, \tilde{N}(\mathrm{d} s, \mathrm{d} z).
\end{align}
\end{definition}

\section{Existence and uniqueness}\label{sec-lcoal-ex}
\begin{lemma}\label{lem-A-1/2-s}

For $ \bm{u}_1\in \mathbb{L}^{p\otimes}_{\sigma}(D),\; \bm{u}_2\in \mathbb{L}^{p\otimes}(D)$, $1<p<\infty$,  there exists a contant $C_p>0$ such that
\begin{align}
      \| \bmA^{-\frac12} \bm{B}(\bm{u}_1,\bm{u}_2)\|_{\mathbb{L}^{\frac{p}2\otimes}(D)}&\leq C_p\,\|\bm{u}_1\|_{\mathbb{L}^{p\otimes}(D)}\|\bm{u}_2\|_{\mathbb{L}^{p\otimes}(D)},\label{Lemma-A-1/2-eq-1}\\
			 \| \bmA^{-\frac12} \bm{B}(\bm{u}_1,\bm{u}_1)-\bmA^{-\frac12} \bm{B}(\bm{u}_2,\bm{u}_2)\|_{\mathbb{L}^{\frac{p}2\otimes}(D)}
			 & \leq  C_p\,\|\bm{u}_1-\bm{u}_2\|_{\mathbb{L}^{p\otimes}(D)}\Big(\|\bm{u}_1\|_{\mathbb{L}^{p\otimes}(D)}+\|\bm{u}_2\|_{\mathbb{L}^{p\otimes}(D)}\Big).\label{Lemma-A-1/2-eq-2}
		\end{align}
\end{lemma}
\begin{proof}
The proof follows immediately from the fact that $(f \cdot \nabla) g=\sum_{j=1}^d \frac{\partial}{\partial x_j}(f_jg)$, if $\operatorname{div} f=0$ and by \cite[Lemma 2.1]{Giga+Miya}, $A^{-\frac12} P \frac{\partial}{\partial x_j}$ and  $\mathcal{M}^{-\frac12} \frac{\partial}{\partial x_j}$ can be extended uniquely to bounded linear operators from $\mathbb{L}_{\sigma}^{p/2}(D)$ to $\mathbb{L}_{\sigma}^{p/2}(D)$. Recall that if $D=\RR^d$,  $P$ commutes with $\frac{\partial}{\partial x_j}$ and $-\Delta$, and $A=-P\Delta$ is essentially identical with $-\Delta$, so \eqref{Lemma-A-1/2-eq-1} and \eqref{Lemma-A-1/2-eq-2} still hold.

\end{proof}
\begin{theorem}\label{main}
Let $d=2,3$ and $r\geq d$. Assume that $\bm{u}_0 \in \mathbb{L}^{r\otimes}_{\sigma}(D)$ 
and  Assumption \ref{assu-xi} holds. 
 Let $\frac{1}{q}=(\frac{1}{r}-\frac{1}{p})\frac{d}{2}$,  $r<p<\infty $. Then there exists a unique local mild solution $\bm{u} \in L^q(0,T;\mathbb{L}_{\sigma}^{p\otimes}(D))$ to equation \eqref{MHD-eq-main}.
\end{theorem}
\begin{proof}
\textbf{Step 1.} Consider the following stochastic  process
$\bm{Z}(t)=\int_{0}^{t}\int_Z \bm{S}(t) \xi(s,z) \tilde{N}(\mathrm{d} s,\mathrm{d} z)$. 
Since $\mathbb{L}^{r\otimes}_{\sigma}(D)$, $r\geq d\geq 2$ is a martingale type $2$ Banach space,  by using the Burkholder inequality and the boundedness of the semigroup, we have, for any $t\in[0,T]$,
	\begin{align*}
	\mathbb{E}\|Z(t)\|_{\mathbb{L}^{r\otimes}_{\sigma}(D)}
	&\leq C_{r}\mathbb{E}\, \Big(\int_0^t\int_Z\|\bm{S}(t) \xi(s,z) \|_{\mathbb{L}_{\sigma}^{r\otimes}(D)}^2 \nu( \mathrm{d}z)\mathrm{d}s\Big)^\frac{1}{2}\\
	&\leq C_{r}\mathbb{E}\, \Big(\int_0^t\int_Z \|\xi(s,z)\|_{\mathbb{L}^{r\otimes}_{\sigma}(D)}^2 \nu(\mathrm{d} z)\mathrm{d} s\Big)^\frac{1}{2}<\infty.
	\end{align*}
It follows that $Z(t)\in \mathbb{L}^{r\otimes}_{\sigma}(D)$, $\mathbb{P}$-a.s.
For an $\mathbb{L}^{r\otimes}_{\sigma}(D)$-valued predictable process $\xi$, we define an $L^q(0, T ; \mathbb{L}_{\sigma}^{p\otimes}(D))$-valued process
 $\psi_{s,z}:=\big\{[0, T] \ni t \mapsto 1_{[s, T]}(t) \bm{S}(t-s)\xi(s, \omega,z)\big\}$,  $s \in[0, T],\,\omega\in\Omega,\, z \in Z.$ By using Lemma \ref{lemma-est-u_0}, we infer
	\begin{align}\label{proof-main-theo-psi-eq1}
	\|\psi_{s,z}(\cdot)\|_{L^q(0,T;\mathbb{L}^{p\otimes}(D))}^q & =\int_s^T\left\|\bm{S}(t-s) \xi(s, z)\right\|_{\mathbb{L}_{\sigma}^{p\otimes}(D)}^q \mathrm{d} t \nonumber\\
	 &=\int_0^{T-s}\left\|\bm{S}(t) \xi(s, z)\right\|_{\mathbb{L}_{\sigma}^{p\otimes}(D)}^q  \mathrm{d} t\nonumber\\
	&\leq C_{p,q,r} \|\xi(s,z) \|_{\mathbb{L}^{r\otimes}(D)}^q.
		\end{align}
Note that $L^q(0, T ;\mathbb{L}_{\sigma}^{p\otimes}(D)\big)$ is a martingale type 2 Banach space. By applying Burkholder's inequality in Banach space \cite[Theorem 3.1]{Zhu+Brz+Liu-2019} and \eqref{proof-main-theo-psi-eq1} we find 
\begin{align}
	\mathbb{E}\|Z(t)\|_{L^q(0,T;\mathbb{L}_{\sigma}^{p\otimes}(D))}&=\mathbb{E}\Big\| \int_{0}^{t}\int_Z \bm{S}(t) \xi(s,z) \tilde{N}(\mathrm{d} s,\mathrm{d} z)    \Big\|_{L^q(0,T;\mathbb{L}_{\sigma}^{p\otimes}(D))}\nonumber\\
	&=
	\mathbb{E}\Big\| \int_{0}^{t}\int_Z \psi_{s,z}(\cdot) \tilde{N}(\mathrm{d} s,\mathrm{d} z)    \Big\|_{L^q(0,T;\mathbb{L}_{\sigma}^{p\otimes}(D))}\nonumber\\
	&\leq C_{p}\mathbb{E} \Big(\int_0^T\int_Z \|\psi_{s,z}(\cdot)\|^2_{L^q(0,T;\mathbb{L}_{\sigma}^{p\otimes}(D))}\nu(\mathrm{d} z)\mathrm{d} s\Big)^{\frac12}\nonumber\\
	&\leq C_{p,q,r} \mathbb{E} \Big(\int_0^T\int_Z \|\xi(s,z)\|^2_{\mathbb{L}_{\sigma}^{r\otimes}(D)}\nu(\mathrm{d} z)\mathrm{d} s\Big)^{\frac{1}{2}}<\infty.\label{est-con-Z}
	\end{align}
That is $\bm{Z}\in L({\Omega};L^{q}(0,T;\mathbb{L}_{\sigma}^p(D)))$. Meanwhile, by using the above estimate and the Lebesgue dominated convergence theorem, we infer $\lim_{T'\rightarrow 0}\mathbb{E}\|Z(t)\|_{L^q(0,T';\mathbb{L}_{\sigma}^{p\otimes}(D))}=0 $. 

\textbf{Step 2.} 
Set $\bm{Y}:=\bm{u}-\bm{Z}$, with $\bm{u}$ being the solution to equation \eqref{MHD-eq-main}. Define an operator
	\begin{align*} \Gamma(\bm{Y}) := \bm{S}(t)\bm{u}_0 + \int_{0}^{t}\bm{S}(t-s)\bm{B}\big(\bm{Y}(s)+\bm{Z}(s),\bm{Y}(s)+\bm{Z}(s)\big)\mathrm{d} s,\quad  t\in [0,T],
	\end{align*}
	where $\bm{Y}\in L^q(0,T;\mathbb{L}_{\sigma}^{p\otimes}(D))$.  To establish the existence of mild solutions to \eqref{MHD-eq-main}, it is equivalent to finding a fixed point for the operator $\Gamma$ in $L^{q}(0,T;\mathbb{L}^{p\otimes}_{\sigma}(D))$. 
	We shall look for a random time $\tau$  such that $\Gamma(B^\tau_R)\subset B^\tau_R$ and $\Gamma$ is a contraction map on a closed ball $B_R^\tau$ in $L^q(0,\tau;\mathbb{L}_{\sigma}^p(D))$. 
	
Let $r\geq d$.  Since by Step 1, $\lim_{T'\rightarrow 0}\mathbb{E}\|Z(t)\|_{L^q(0,T';\mathbb{L}_{\sigma}^{p\otimes}(D))}=0 $,  so there exists a sequence $\{T_n\}\subset[0,T]$ such that 
     $\lim_{T_n \rightarrow0} \| Z(\cdot)\|_{L^q(0,T_n;\mathbb{L}_{\sigma}^p(D))}=0,$ $\mathbb{P}$-a.s. 
Hence for every $\varepsilon>0$, there exists $T_{\varepsilon}$ such that $ \| Z(\cdot)\|_{L^q(0,T_{\varepsilon};\mathbb{L}_{\sigma}^p(D))}<\varepsilon$, $ \mathbb{P}$-a.s. 
	In view of Lemma \ref{lemma-est-u_0}, we have
$ \left\|\bm{S}(t)\bm{u}_0\right\|_{L^q\left(0,T; \mathbb{L}_{\sigma}^p(D)\right)}\leq C_0\left\| \bm{u}_0\right\|_{\mathbb{L}^r},$ 
hence $ \lim_{T'\rightarrow 0} \||\bm{S}(t)\bm{u}_0\|_{L^q(0,T'; \mathbb{L}_{\sigma}^p(D))}=0$. 
Therefore, 
we can always find  $0<T_0<T$ and $\Omega_0 \in \mathcal{F}$ with $\mathbb{P}(\Omega_0)=1$ such that for  $d\leq r<p$,
\begin{align*}
        \|\bm{S}(t)\bm{u}_0\|_{L^q(0,T_0;\mathbb{L}_{\sigma}^{p\otimes}(D))}+  \|\bm{Z}(\cdot)\|_{L^q(0,T_0;\mathbb{L}^{p\otimes}_{\sigma}(D))}<(10 T_0^{\frac12-\frac{d}{2r}} K_1)^{-1},
\end{align*}
where the constant $K_1$ depends on $d,r,p,q$ and will be determined later.

If $\omega\in\Omega_0$, define $\tau(\omega)=T_0$, if $\omega\notin\Omega_0$, define $\tau(\omega)=0$. Define a random variable 
	$N_\tau=\|\bm{S}(t)\bm{u}_0\|_{L^q(0,\tau;\mathbb{L}_{\sigma}^{p\otimes}(D))}+\|Z(t)\|_{L^q(0,\tau;\mathbb{L}_{\sigma}^{p\otimes}(D))}.$
	Let $B_R^\tau$ denote the closed ball in $L^q(0,\tau;\mathbb{L}_{\sigma}^p(D))$ with radious $R=2N_\tau$.  Next we shall prove $\Gamma( B_{R}^{\tau})\subset B_{R}^{\tau}$. 
	Let $\bm{Y}\in B_{R}^{\tau}$. By using \eqref{ineq-semi-pro-1-MHD} and Lemma \ref{lem-A-1/2-s}, we have
		\begin{align*}
&\Big\|   \int_{0}^{t}\bm{S}(t-s) \bm{B}\big(\bm{Y}(s)+\bm{Z}(s),\bm{Y}(s)+\bm{Z}(s)\big)\mathrm{d}s   \Big\|_{\mathbb{L}^{p\otimes}_{\sigma}(D)} \\
&\leq \int_0^t \Big\| \bmA^{\frac12}  \bm{S}\big(\frac{t-s}2\big) \bm{S}\big(\frac{t-s}2\big)\bmA^{-\frac12}  \bm{B}\big(\bm{Y}(s)+\bm{Z}(s),\bm{Y}(s)+\bm{Z}(s)\big) \Big\|_{\mathbb{L}^{p\otimes}_{\sigma}(D)}\mathrm{d}s\\
&\leq C_{d,p}  \int_0^t \Big(\frac{t-s}2\Big)^{-\big(\frac12+\frac{d}{2p}\big) } \big\|  \bmA^{-\frac12} \bm{B}\big(\bm{Y}(s)+\bm{Z}(s),\bm{Y}(s)+\bm{Z}(s)\big) \big\|_{\mathbb{L}^{\frac{p}2\otimes}_{\sigma}(D)}\mathrm{d}s\\
&\leq C_{d,p}\int_0^t (t-s)^{-(\frac12+\frac{d}{2p}) } \|\bm{Y}(s)+\bm{Z}(s)\|^2_{\mathbb{L}_{\sigma}^{p\otimes}(D)}\mathrm{d} s\\
&\leq C_{d,p}\tau^{\frac12-\frac{d}{2r}}\int_0^t (t-s)^{-\big(1-\frac{d}2(\frac1r-\frac{1}{p})\big)} \|\bm{Y}(s)+\bm{Z}(s)\|^2_{\mathbb{L}_{\sigma}^{p\otimes}(D)}\mathrm{d} s.
\end{align*}
Noticing  $ \frac{1}{q}= \frac{d}{2}\big(\frac{1}{r}-\frac{1}{p}\big)$ and
applying the Hardy-Littlewood-Sobolev inequality (c.f. \cite{Stein}), we obtain  
		\begin{align*}
&\Big\|   \int_{0}^{t}\bm{S}(t-s) \bm{B}\big(\bm{Y}(s)+\bm{Z}(s),\bm{Y}(s)+\bm{Z}(s)\big)\mathrm{d} s   \Big\|_{L^q(0,\tau;\mathbb{L}_{\sigma}^{p\otimes}(D))}\\
&\leq C_{d,p}K_{r,p,q}\tau^{\frac12-\frac{d}{2r}} \|\bm{Y}(s)+\bm{Z}(s)\|^2_{L^q(0,\tau;\mathbb{L}_{\sigma}^{p\otimes}(D))}\\
&\leq K_1\tau^{\frac12-\frac{d}{2r}} \Big(\|\bm{Y}(s)\|_{L^q(0,\tau;\mathbb{L}_{\sigma}^{p\otimes}(D))}+\|\bm{Z}(s)\|_{L^q(0,\tau;\mathbb{L}_{\sigma}^{p\otimes}(D))}\Big)^2,
\end{align*}	
with $K_1=C_{d,p}K_{r,p,q}$ and $K_{r,p,q}$ being the constant from the Hardy-Littlewood-Sobolev inequality.
It follows that 
\begin{align*}
 \|\Gamma (\bm{Y})\|_{L^q(0,\tau;\mathbb{L}^{p\otimes}_{\sigma}(D))} &\leq  \|\bm{S}(t)\bm{u}_0\|_{L^q(0,T_0;\mathbb{L}_{\sigma}^{p\otimes}(D))} +   K_1 \tau^{\frac12-\frac{d}{2r}}(\|\bm{Y}\|_{L^q(0,\tau;\mathbb{L}^{p\otimes}_{\sigma}(D)}+\|\bm{Z}\|_{L^q(0,\tau;\mathbb{L}_{\sigma}^{p\otimes}(D))})^2\\
 &<N_{\tau}+ K_1\tau^{\frac12-\frac{d}{2r}} 9N_{\tau}\cdot N_{\tau}<2N_{\tau}.
\end{align*}
Hence we obtain $\Gamma(Y)\in B_R^\tau$. 

Next we prove that $\Gamma$ is a contraction map in $B_R^\tau$. Let $\bm{Y}_1,\bm{Y}_2\in B_R^\tau$. Applying similar arguments as before, we have
\begin{align*}
    &\|\Gamma(\bm{Y}_1)-\Gamma(\bm{Y}_2)\|_{\mathbb{L}^p_{\sigma}(D)}\\
    &\leq \Big\| \int_{0}^{t}\bm{S}(t-s)\bm{B}(\bm{Y}_1-\bm{Y}_2,\bm{Y}_1+\bm{Z}) \mathrm{d} s \Big\|_{\mathbb{L}^{p\otimes}_{\sigma}(D)}+\Big\| \int_{0}^{t} \bm{B}(\bm{Y}_2+\bm{Z},\bm{Y}_1-\bm{Y}_2)\mathrm{d} s\Big\|_{\mathbb{L}^{p\otimes}_{\sigma}(D)}\\
          &\leq C_{d,p}\tau^{\big(\frac12-\frac{d}{2r}\big)} \int_0^t  (t-s)^{-\big(1-\frac{d}2(\frac1r-\frac{1}{p})\big)} \big(\|\bm{Y}_1\|_{\mathbb{L}^{p\otimes}_{\sigma}(D)}+ \|\bm{Y}_2\|_{\mathbb{L}^{p\otimes}_{\sigma}(D)}+2 \|\bm{Z}\|_{\mathbb{L}^{p\otimes}_{\sigma}(D)}\big)
          \|\bm{Y}_1-\bm{Y}_2\|_{\mathbb{L}^{p\otimes}_{\sigma}(D)} \mathrm{d} s .
\end{align*}
By using again the Hardy-Littlewood-Sobolev inequality, we infer
\begin{align*}
   & \|\Gamma(\bm{Y}_1)(t)-\Gamma(\bm{Y}_2)(t)\|_{L^q(0,\tau;\mathbb{L}_{\sigma}^p(D))}\\
   &\leq  C_{d,p}K_{p,r,q}\tau^{\frac12-\frac{d}{2r}}\big(\|\bm{Y}_1\|_{L^q(0,\tau;\mathbb{L}_{\sigma}^{p\otimes}(D))}+ \|\bm{Y}_2\|_{L^q(0,\tau;\mathbb{L}^{p\otimes}_{\sigma}(D))}+2\|\bm{Z}\|_{L^q(0,\tau;\mathbb{L}_{\sigma}^{p\otimes}(D))}\big)
   \|\bm{Y}_1-\bm{Y}_2\|_{L^q(0,\tau;\mathbb{L}_{\sigma}^{p\otimes}(D))}\\
   &\leq K_1\tau^{\frac12-\frac{d}{2r}}(6N_{\tau})\|\bm{Y}_1-\bm{Y}_2\|_{L^q(0,\tau;\mathbb{L}^{p\otimes}_{\sigma}(D))}\\
   &<\frac35\|\bm{Y}_1-\bm{Y}_2\|_{L^q(0,\tau;\mathbb{L}_{\sigma}^p(D))}.
\end{align*}

{\textbf{Step 3. Uniqueness}} Next we prove the uniqueness of the local solution.  Let $\bm{u}_1,\bm{u}_2$ be two local mild solutions  to equation \eqref{MHD-eq-main} given by \eqref{local solution}. Applying similar arguments as in Step 2 gives 
\begin{align*}
\Vert \bm{u}_1-\bm{u}_2\Vert_{\mathbb{L}_{\sigma}^{p\otimes}(D)}&=\Big\| \int_{0}^{t}\bm{S}(t-s)(B(\bm{u}_1(s),\bm{u}_1(s))-B(\bm{u}_2(s),\bm{u}_2(s)))\mathrm{d} s\Big\|_{L^{p\otimes}(D)}\\
&\leq  C_{d,p}\tau^{\frac12-\frac{d}{2r}}\int_0^t (t-s)^{-\big(1-\frac{d}2(\frac1r-\frac{1}{p})\big)}\| \bm{u}_1-\bm{u}_2\|_{\mathbb{L}^{p\otimes}_{\sigma}(D)}\big(\| \bm{u}_1\|_{\mathbb{L}^{p\otimes}_{\sigma}(D)}+\| \bm{u}_2\|_{\mathbb{L}^{p\otimes}_{\sigma}(D)}\big)\, \mathrm{d}s.
\end{align*}
Applying again the Hardy-Littlewood-Sobolev inequality gives
\begin{align*}
	\|\bm{u}_1(t)-\bm{u}_2(t)\|_{L^q(0,\tau;\mathbb{L}_{\sigma}^{p\otimes}(D))}
	&\leq K_1\tau^{\frac12-\frac{d}{2r} }\|\bm{u}_1-\bm{u}_2\|_{L^q(0,\tau;\mathbb{L}_{\sigma}^{p\otimes}(D))}\Big(\|\bm{u}_1\|_{L^q(0,\tau;\mathbb{L}_{\sigma}^{p\otimes}(D))} + \|\bm{u}_2\|_{L^q(0,\tau;\mathbb{L}_{\sigma}^{p\otimes}(D))}\Big)\\
	&\leq K_1\tau^{\frac12-\frac{d}{2r} }(4N_\tau)\|\bm{u}_1-\bm{u}_2\|_{L^q(0,\tau;\mathbb{L}_{\sigma}^{p\otimes}(D))}\\
	& \leq \frac{2}{5} \|\bm{u}_1-\bm{u}_2\|_{L^q(0,\tau;\mathbb{L}_{\sigma}^{p\otimes}(D))},
\end{align*}
This shows that $\bm{u}_1=\bm{u}_2$ in $L^q(0,\tau;\mathbb{L}_{\sigma}^{p\otimes}(D))$.

\end{proof}

\begin{theorem}\label{Th-2D}
Let $d=2$.  Assume that $\bm{u}_0 \in \mathbb{L}_{\sigma}^{2\otimes}(D)$ and Assumption \ref{assu-xi} holds with $r=2$. 
Let $\frac1q=\frac12-\frac1p$, $p>2$. Then there exists a unique global mild solution $\bm{u} \in L^q(0,T;\mathbb{L}_{\sigma}^{p\otimes}(D))$ to equation \eqref{MHD-eq-main}.
\end{theorem}

\begin{proof}
The local existence follows from Theorem \ref{main}. Now let us establish an a priori bound that ensures global existence.
 We shall show that the $\mathbb{L}^{2\otimes}_{\sigma}(D)$-norm of $\bm{Y}$ is bounded on the interval $\left[0, T_0\right]$. Using the chain rule gives
\begin{align*}
\frac{1}{2} \frac{\mathrm{d}}{\mathrm{d} t}\|\bm{Y}(t)\|_{\mathbb{L}^{2\otimes}(D)}^2 
& =-\|\nabla \bm{Y}(t)\|_{\mathbb{L}^{2\otimes}(D)}^2 -b({Y}_1(t)+{Z}_1(t),{Z}_1(t),{Y}_1(t) )+b({Y}_2(t)+{Z}_2(t),{Z}_2(t),{Y}_1(t) )\\
&\quad-b({Y}_1(t)+{Z}_1(t),{Z}_2(t),{Y}_2(t) )+b({Y}_2(t)+{Z}_2(t),{Z}_1(t),{Y}_2(t) ),
\end{align*}
where $\bm{Y}=({Y}_1,{Y}_2), \bm{Z}=({Z}_1,{Z}_2)$. By using the property of $\bm{B}$, \eqref{est-tri-b-1} and Young's inequality, we obtain
\begin{align*}
\frac{1}{2} \frac{\mathrm{d}}{\mathrm{d} t}\|\bm{Y}(t)\|_{\mathbb{L}^{2\otimes}(D)}^2 
& \leq -\|\nabla \bm{Y}(t)\|_{\mathbb{L}^{2\otimes}(D)}^2+ 4\|\bm{Y}(t)+\bm{Z}(t)\|_{\mathbb{L}^{4\otimes}(D)}\|\bm{Z}(t)\|_{\mathbb{L}^{4\otimes}(D)}\|\nabla \bm{Y}\|_{\mathbb{L}^{2\otimes}(D)}\\
&\leq -\frac12 \|\nabla \bm{Y}(t)\|_{\mathbb{L}^{2\otimes}(D)}^2 +C_1\| \bm{Z}\|_{\mathbb{L}^{4\otimes}(D)}^4 \|\bm{ Y}\|_{\mathbb{L}^{2\otimes}(D)}^2+C_2\| \bm{Z}\|_{\mathbb{L}^{4\otimes}(D)}^4.
\end{align*}
Now applying the Gronwall's Lemma yields that
\begin{align}\label{uniform-est-pro}
\sup_{t\in[0,T]}\|\bm{Y}(t)\|_{\mathbb{L}^{2\otimes}_{\sigma}(D)}^2\leq C_2 \int_0^T e^{C_1 \int_s^T\|\bm{Z}(r)\|_{\mathbb{L}^{4\otimes}(D)}^4 \mathrm{dr}}\|\bm{Z}(s)\|_{\mathbb{L}^{4\otimes}(D)}^4 \mathrm{~d} s<\infty ,\quad \mathbb{P}\text{-a.s.}
\end{align}
where we used \eqref{est-con-Z} with $p=q=4$ to infer $\mathbb{E}\|\bm{Z}\|_{L^4(0,T; \mathbb{L}^{4\otimes}(D))}<\infty$. 
  Since $\bm{Y}_0\in \mathbb{L}_{\sigma}^{2\otimes}(D)$, by Theorem \ref{main} there exists a unique local solution $\bm{Y}$ which belongs to $L^{q}(0,\tau;\mathbb{L}_{\sigma}^{p\otimes }(D))$. 
   By the above priori estimate \eqref{uniform-est-pro}, we infer $\bf{Y}(\tau)\in \mathbb{L}_{\sigma}^{2\otimes}(D)$. We can repeating the argument above and consider equation \eqref{MHD-eq-main} with initial data $\bm{Y}(\tau)$:
    \begin{align}
    \begin{split}\label{MHD-new-tau}
 \mathrm{d} \bm{v} (t) =-\bmA \bm{v}(t)\mathrm{d} t+\bm{B}(\bm{v}(t)+\bm{Z}^{\tau}(t),\bm{v}(t)+\bm{Z}^{\tau}(t))\mathrm{d}t,\quad
 \bm{v}(0)=\bm{Y}(\tau),
 \end{split}
 \end{align}
where 
$
\bm{Z}^{\tau}(t)=\int_{0}^t \int_Z \xi(s, z) \tilde{N}^{\tau}(\mathrm{d} s, \mathrm{d} z), \text{ with } N^{\tau}(t,A):=N(t+\tau,A)-N(\tau,A),
$
for each $t\geq0$ and $A\in \mathcal{Z}$. It is easy to verify that $N^{\tau}(t,A)$ is a Poisson random measure with respect to  the shifted filtration $\mathbb{F}^{\tau}:=(\mathcal{F}_{t+\tau})_{t\geq0}$ with the same intensity measure $\nu$ and 
$\int_{0}^{t}\int_{Z}\bm{S}(t-s)\xi(\tau+s,z)\,\tilde{N}^{\tau}(\d s,\d z)=\int_{\tau}^{\tau+t}\int_{Z}\bm{S}(\tau+t-s)\xi(s,z)\,\tilde{N}(\d s,\d z).$ 
By Theorem \ref{main}, Equation \eqref{MHD-new-tau} has a unique solution $v\in  L^{q}(\tau,  \tau_1\;\mathbb{L}_{\sigma}^{p\otimes}(D)) $ with $\tau_1>\tau$. In this way we extend our solution to the time interval $\left[\tau, \tau_1\right]$. Repeating this procedure finitely many times leads to solutions on $[0,\tau_{\infty}]$, where $\tau_{\infty}$ is the supremum time over $[0,T]$ upon which the solution exists. 

 Note that by interpolation theorem and the Sobolev embedding we have
$L^{\infty}(0, T ; \mathbb{L}_\sigma^{2 \otimes}(D)) \cap L^2(0, T ; \mathbb{H}^{1,2\otimes}(D)) \subset L^{q}(0, T ; \mathbb{H}^{1-\frac{2}{p},2\otimes}(D)) \hookrightarrow L^{q}(0, T ; \mathbb{L}_\sigma^{p \otimes}(D))$, for $2\leq p<\infty$ with $\frac1q=\frac12-\frac1p$. Then by \eqref{uniform-est-pro}, we have the a priori estimate for $Y$ in  $L^{q}(0, T ; \mathbb{L}_\sigma^{p \otimes}(D))$.  By following a standard contradiction argument, the above uniform priori estimate leads to $\tau_{\infty}=T$. Therefore we construct a global mild solution existing on the full interval $[0, T]$.

\end{proof}

\end{document}